\documentclass[10pt]{article}

\usepackage{graphicx}
\usepackage{amssymb}
\usepackage{epstopdf}
\usepackage{tikz}

\usepackage{amsfonts}

\usepackage{amsmath}
\usepackage[colorlinks=true,
            linkcolor=blue,
        citecolor=blue,
            urlcolor=blue]{hyperref}
\usepackage[nameinlink,noabbrev]{cleveref}
\usepackage{amssymb}
\usepackage{comment}
\usepackage{color}
\usepackage{amsthm}

\def\phi{{\varphi}}

\DeclareSymbolFont{AMSb}{U}{msb}{m}{n}

\usepackage{indentfirst}
\begin{document}

\addtolength{\textheight}{0 cm}
\addtolength{\hoffset}{0 cm}
\addtolength{\textwidth}{0 cm}
\addtolength{\voffset}{0 cm}

\setcounter{secnumdepth}{5}

\newtheorem{theorem}{Theorem}[section]
\newtheorem{lemma}[theorem]{Lemma}
\newtheorem{coro}[theorem]{Corollary}
\newtheorem{remark}[theorem]{Remark}
\newtheorem{ex}[theorem]{Example}
\newtheorem{claim}[theorem]{Claim}
\newtheorem{conj}[theorem]{Conjecture}
\newtheorem{definition}[theorem]{Definition}
\newtheorem{proposition}[theorem]{Proposition}
\newtheorem{application}{Application}

\newtheorem{corollary}[theorem]{Corollary}
\def\pR{\R\cup\{+\infty\}}

\newcommand{\ntx}{\textnormal}
\newcommand\N{{\mathbb N}}
\newcommand\Z{{\mathbb Z}}
\newcommand\Q{{\mathbb Q}}
\newcommand\R{{\mathbb R}}
\newcommand\E{{\mathbb E}}
\newcommand\mL{{\mathcal{L}}}
\newcommand{\mB}{\mathcal{B}}
\newcommand{\tow}{\rightharpoonup}	
\newcommand{\toe}{\hookrightarrow}	
\def\F{{\mathcal {F}}}
\def\e{\epsilon}
\def\P{{\mathbb P}}
\def\L{{\mathcal {L}}}
\def\X{{\mathcal {X}}}
\def\div{\textnormal{div} }
\def\inn{\textnormal{in} }
\def\on{\textnormal{on} }
\def\msX{\mathscr{X}}

\def\msE{\mathscr{E}}
\def\msM{\mathscr{M}}
\def\msH{\mathscr{H}}
\def\msF{\mathscr{F}}
\def\msL{\mathcal{L}}
\def \mbL{\mathbb{L}}
\def\vr{\vert}
\def\Vr{\Vert}

\def\F{{\cal F}}
\def\G{{\cal G}}
\def\L{{\cal L}}
\def\Rm{{\cal R}}
\def\div{\mbox{div} }
\def\t{\tilde}
\title{A Fractional Logistic-Type Elliptic Problem with Global Interactions:
Existence, Local Uniqueness, and the Fractional-to-Local Limit }
\author{ Alireza Khatib\thanks{Corresponding author: Alireza Khatib. Universidade Federal do Amazonas, Manaus-AM, Brazil.
Email: \texttt{alireza@ufam.edu.br}}  \quad  Somayeh Mousavinasr \thanks{Universidade Federal do Amazonas, Manaus-AM, Brazil, Somayeh@ufam.edu.br
 } \quad  Bashir Zeimarani \thanks{Instituto CERTI Amazônia (ICA), Manaus, AM, Brazil, bsz@certi.org.br
 }  }

\date{}

 \maketitle
 \begin{abstract}
We study a fractional logistic-type elliptic problem on a bounded domain
with homogeneous exterior conditions and linear and nonlinear integral
interactions. The interaction kernels may be nonsymmetric and
sign-changing, so the problem need not admit a variational formulation.
In the symmetric competitive regime, coercivity yields a weak solution
by global minimization. For arbitrary fixed values of the interaction
parameter, explicit smallness conditions allow us to apply Schauder's
fixed-point theorem without symmetry or sign restrictions on the kernels.
Additional sign assumptions give nonnegative solutions. When the
nonlinear exponent satisfies $q\geq2$, a contraction condition yields
uniqueness in an invariant ball and convergence of the Picard iteration.
Finally, as the fractional order approaches one, we prove subsequential
convergence to a solution of the local Dirichlet problem, together with
convergence of the fractional energies to the Dirichlet energy.
Uniqueness of the local solution in the limiting ball gives convergence
of the entire family.
\end{abstract}


\medskip

\section{Introduction}\label{sec:introduction}

Nonlocal diffusion models arise naturally in the description of long-range
dispersal, anomalous transport, and heterogeneous spatial dynamics. In such
models, particles or individuals may perform jumps over arbitrary distances,
leading to integro-differential equations rather than classical local
diffusion equations. With the usual positive-operator convention,
$-(-\Delta)^s$ is the infinitesimal generator of a symmetric
$2s$-stable L\'evy process in $\mathbb R^N$, while the exterior condition
$u=0$ in $\mathbb R^N\setminus\Omega$ corresponds to killing the process
upon leaving $\Omega$. We refer to \cite{ChenKimSong2010} for the
probabilistic viewpoint. The functional framework for fractional Sobolev
spaces and their embeddings is developed in
\cite{DiNezzaPalatucciValdinoci2012}, whereas variational methods for
nonlocal Dirichlet problems are systematically studied in
\cite{ServadeiValdinoci2013}. See also
\cite{CaffarelliSilvestre2007,RosOtonSerra2014} for fundamental results
on fractional elliptic equations and exterior Dirichlet conditions.

Nonlocal effects may also enter through the interaction terms. In population
models, for instance, the growth or dispersal rate at a given point may
depend on the population distribution throughout the habitat. Such effects
lead naturally to integral operators and to equations combining local or
nonlocal diffusion with global interactions. A broad class of nonlocal
elliptic problems of this type has been studied by several authors; see,
among others,
\cite{ChipotRodrigues1992,chipot2006remarks,
AlvesCorreaChipot2017}.
Related logistic models with nonlocal diffusion and population-dependent
interaction mechanisms have also been investigated using bifurcation,
approximation, and fixed-point techniques; see, for example,
\cite{CintraOliveiraSantosSuarez2026}.

Fractional logistic-type equations provide another important class of
nonlocal population models; see, for example,
\cite{MontefuscoPellacciVerzini2013,QuaasXia2017,
IannizzottoMosconiPapageorgiou2023}. Nonlocal dispersal models involving
integral kernels were studied in
\cite{HutsonMartinezMischaikowVickers2003}, with further developments
concerning persistence, spectral quantities, and spreading phenomena in
\cite{Coville2006,Coville2010,ShenXie2016}. Nonlinear global interactions
are also closely related to Hartree--Choquard-type terms; see
\cite{MorozVanSchaftingen2017}.

Motivated by these developments, we study the fractional elliptic problem
\begin{equation}\label{P}
\begin{cases}
(-\Delta)^s u
-\displaystyle\int_{\Omega}K(x,y)u(y)\,dy
=
a(x)u
+\lambda |u|^{q-2}u
\displaystyle\int_{\Omega}Q(x,y)|u(y)|^q\,dy
+f(x),
& x\in\Omega,\\[0.3cm]
u=0,
& x\in\mathbb{R}^{N}\setminus\Omega.
\end{cases}
\end{equation}
Here $\Omega\subset\mathbb R^N$ is a bounded domain, $0<s<1$,
$a$ is a bounded potential, and $K$ and $Q$ are bounded interaction
kernels. The operator associated with $K$ describes a linear global
interaction, whereas the term involving $Q$ has effective order $2q-1$
and couples the value of the solution at $x$ with its distribution over
the entire domain.

Problem~\eqref{P} therefore combines two distinct sources of nonlocality:
the fractional diffusion operator and the global integral interactions.
Their simultaneous presence creates difficulties that are absent when
either mechanism is considered separately. The linear interaction
associated with $K$ may affect the coercivity of the fractional operator,
while the nonlinear term involving $Q$ is not variational in general.
Indeed, if $K$ or $Q$ is nonsymmetric, the corresponding operator need
not arise as the derivative of a scalar functional; moreover, for
sign-changing kernels, the interactions cannot be assigned a fixed
cooperative or competitive character. These features place
problem~\eqref{P} outside the standard variational framework used for many
fractional logistic and Choquard-type equations.

\paragraph*{Relation with the local problem and main contributions.}
A local counterpart of~\eqref{P}, with the classical Laplacian in place
of the restricted fractional Laplacian, was studied in
\cite{KhatibOliveiraDosAnjosFurtadoCacau2026}. The minimax mechanism
employed there builds on the broader variational framework developed in
\cite{KhatibMoameniMousavinasr2026}.The fractional setting involves a homogeneous exterior condition
and a normalized Dirichlet form on spaces depending on $s$.
Here we develop an existence theory for nonsymmetric and
sign-changing interaction kernels, obtain uniqueness under an
explicit contraction condition, and establish convergence of
solutions and energies as $s\uparrow1$. 

The approach developed in this paper separates the analysis into two
regimes. In the symmetric competitive case, the fractional energy and an
explicit coercivity condition lead to a direct variational argument. In
the general case, where symmetry and sign assumptions on the kernels are
dropped, we instead construct a compact solution operator associated with
an auxiliary fractional Dirichlet problem and apply Schauder's fixed-point
theorem. This formulation allows us to treat genuinely nonvariational
interactions within the fractional setting.

Beyond the existence theory, two further conclusions distinguish the
present analysis from its local counterpart. First, under a stronger
smallness condition, the solution operator becomes a contraction on an
invariant ball. This yields uniqueness within that ball and strong
convergence of the corresponding Picard iterates. Second, we establish a
normalized fractional-to-local limit as $s\to1^-$, showing that suitable
families of fractional solutions converge, up to subsequences, to
solutions of the associated local elliptic problem. 

\subsection{Main results}

We establish four complementary results for problem~\eqref{P}.

\begin{enumerate}

\item \emph{Competitive symmetric regime.}
Assume that $\lambda<0$, the kernels $K$ and $Q$ are symmetric,
$Q\geq0$, and the linear part is coercive. Then problem~\eqref{P}
admits a variational formulation, and the associated energy functional
possesses a global minimizer. Consequently, a weak solution exists and is
nontrivial whenever $f\neq0$. Under the additional assumptions $K\geq0$
and $f\geq0$, a nonnegative weak solution is obtained.

\item \emph{General nonvariational regime.}
Without symmetry or sign assumptions on $K$ and $Q$, we construct
a compact solution operator through an auxiliary fractional
Dirichlet problem. Under an explicit smallness condition on the
linear coefficients, each fixed $\lambda\in\mathbb R$ admits an
invariant-ball construction for sufficiently small forcing terms.
Schauder's fixed-point theorem then provides a weak solution.
Under additional positivity assumptions, nonnegative solutions
are obtained without requiring symmetry of the kernels.

\item \emph{Local uniqueness and Picard convergence.}
If $q\geq2$ and a stronger smallness condition holds, the solution
operator is a contraction on the invariant ball. Hence the weak solution
is unique within that ball, and the corresponding Picard iterates converge
strongly to it.

\item \emph{Fractional-to-local limit.}
For normalized fractional problems with estimates uniform in $s$, every
sequence of solutions with $s\to1^-$ admits a subsequence converging
strongly in $L^p(\Omega)$, for every $1\leq p<2^*$, to a weak solution
of the corresponding local elliptic problem. Moreover, the normalized fractional energies converge to the
Dirichlet energy of the limit. If the local solution is unique
in the limiting ball, the entire family converges.

\end{enumerate}

The paper is organized as follows.
Section~\ref{sec:preliminaries} introduces the fractional functional
setting, the interaction operators, and the weak formulation.
Section~\ref{sec:competitive-regime} treats the symmetric competitive
regime. Section~\ref{sec:fixed-point} develops the fixed-point approach,
the positive subregime, and the local uniqueness result.
Section~\ref{sec:local-limit} establishes the fractional-to-local
convergence. 
\section{Preliminaries}
\label{sec:preliminaries}
Throughout the paper, $\Omega\subset\mathbb R^N$, $N\geq2$, is a
bounded domain with Lipschitz boundary, $s\in(0,1)$, and
\[
1<q<2_s^*:=\frac{2N}{N-2s}.
\]
Notice that $N>2s$ under these assumptions.

\subsection{Fractional Sobolev setting}

The natural space associated with the homogeneous exterior condition
in problem~\eqref{P} is
\[
X_0^s(\Omega)
:=
\left\{
u\in H^s(\mathbb R^N):
u=0 \text{ a.e. in }\mathbb R^N\setminus\Omega
\right\}.
\]
For $u,v\in X_0^s(\Omega)$, define
\begin{equation}\label{eq:fractional-bilinear-form}
\mathcal E_s(u,v)
:=
\frac{C_{N,s}}{2}
\iint_{\mathbb R^{2N}}
\frac{(u(x)-u(y))(v(x)-v(y))}
{|x-y|^{N+2s}}\,dx\,dy,
\end{equation}
where
\[
C_{N,s}
=
\frac{4^s s\,\Gamma(N/2+s)}
{\pi^{N/2}\Gamma(1-s)}
\]
is the standard normalization constant. We equip $X_0^s(\Omega)$ with
the norm
\[
\|u\|_{X_0^s}:=\mathcal E_s(u,u)^{1/2}.
\]
By the fractional Poincar\'e inequality, this norm is equivalent,
for fixed $s$, to the norm inherited from $H^s(\mathbb R^N)$.
Thus $X_0^s(\Omega)$ is a Hilbert space. We denote its dual by
\[
X^{-s}(\Omega):=\bigl(X_0^s(\Omega)\bigr)^*.
\]

The restricted fractional Laplacian with homogeneous exterior
condition is understood weakly through
\[
\langle(-\Delta)^s u,v\rangle=\mathcal E_s(u,v),
\qquad u,v\in X_0^s(\Omega).
\]
For smooth functions, this agrees with the usual principal-value
definition. The normalization above is also consistent with the
local limit:
\[
\lim_{s\uparrow1}\mathcal E_s(u,u)
=
\int_\Omega|\nabla u|^2\,dx,
\qquad u\in C_c^\infty(\Omega).
\]

The embedding $X_0^s(\Omega)\hookrightarrow L^r(\Omega)$ is continuous
for $1\leq r\leq2_s^*$ and compact for $1\leq r<2_s^*$.
In particular,
\begin{equation}\label{eq:fractional-embedding}
\|u\|_{L^r(\Omega)}
\leq C_{S,r}\|u\|_{X_0^s},
\qquad u\in X_0^s(\Omega),\quad 1\leq r\leq2_s^*.
\end{equation}
Here $C_{S,r}$ may depend on $N,s,\Omega$, and $r$; no uniformity
with respect to $s$ is asserted at this stage.
We refer to \cite{DiNezzaPalatucciValdinoci2012} for the standard
fractional Sobolev properties used below.

\subsection{Nonlocal operators and weak solutions}

Let $K,Q\in L^\infty(\Omega\times\Omega)$ and
$a\in L^\infty(\Omega)$. No symmetry or sign assumptions on these
coefficients are imposed in this section. For brevity, $\|K\|_\infty$
and $\|Q\|_\infty$ denote the norms on $\Omega\times\Omega$,
whereas $\|a\|_\infty$ denotes the norm on $\Omega$.

For $u\in L^2(\Omega)$ and $u\in L^q(\Omega)$, respectively, set
\[
(T_Ku)(x):=\int_\Omega K(x,y)u(y)\,dy,
\qquad
B_Q(u)(x):=\int_\Omega Q(x,y)|u(y)|^q\,dy.
\]
The nonlinear interaction is
\[
\mathcal N_Q(u):=|u|^{q-2}u\,B_Q(u),
\]
where $|u|^{q-2}u$ is understood to be zero at $u=0$.

The boundedness of the kernels gives
\begin{equation}\label{eq:TK-estimate}
\|T_Ku\|_{L^2(\Omega)}
\leq |\Omega|\,\|K\|_\infty\|u\|_{L^2(\Omega)}
\end{equation}
and
\begin{equation}\label{eq:BQ-estimate}
\|B_Q(u)\|_{L^\infty(\Omega)}
\leq \|Q\|_\infty\|u\|_{L^q(\Omega)}^q.
\end{equation}
Writing $q'=q/(q-1)$, we consequently obtain
\begin{equation}\label{eq:nonlinear-Lqprime-bound}
\|\mathcal N_Q(u)\|_{L^{q'}(\Omega)}
\leq \|Q\|_\infty\|u\|_{L^q(\Omega)}^{2q-1}.
\end{equation}
H\"older's inequality and \eqref{eq:fractional-embedding} therefore yield
\begin{equation}\label{eq:nonlinear-dual-bound}
\|\mathcal N_Q(u)\|_{X^{-s}(\Omega)}
\leq
\|Q\|_\infty C_{S,q}^{2q}
\|u\|_{X_0^s}^{2q-1}.
\end{equation}
Thus $\mathcal N_Q:X_0^s(\Omega)\to X^{-s}(\Omega)$ is well defined.
In particular, the interaction requires $L^q$ integrability, rather
than $L^{2q}$ integrability, despite its homogeneity of degree $2q-1$.

\begin{lemma}[Continuity and compactness of the interactions]
\label{lem:interaction-compactness}
The map $\mathcal N_Q:L^q(\Omega)\to L^{q'}(\Omega)$ is continuous.
Moreover, if $u_j\rightharpoonup u$ in $X_0^s(\Omega)$, then
\[
au_j\to au,\qquad
T_Ku_j\to T_Ku,\qquad
\mathcal N_Q(u_j)\to\mathcal N_Q(u)
\quad\text{in }X^{-s}(\Omega).
\]
\end{lemma}

\begin{proof}
Suppose first that $u_j\to u$ in $L^q(\Omega)$. The standard
continuity properties of power maps give
\[
|u_j|^q\to|u|^q\quad\text{in }L^1(\Omega),
\qquad
|u_j|^{q-2}u_j\to|u|^{q-2}u
\quad\text{in }L^{q'}(\Omega).
\]
Since
\[
\|B_Q(u_j)-B_Q(u)\|_{L^\infty(\Omega)}
\leq
\|Q\|_\infty
\bigl\||u_j|^q-|u|^q\bigr\|_{L^1(\Omega)},
\]
we have $B_Q(u_j)\to B_Q(u)$ in $L^\infty(\Omega)$.
Taking products proves
$\mathcal N_Q(u_j)\to\mathcal N_Q(u)$ in $L^{q'}(\Omega)$.

Now let $u_j\rightharpoonup u$ in $X_0^s(\Omega)$.
The compact embeddings give strong convergence in both
$L^2(\Omega)$ and $L^q(\Omega)$. Consequently,
$au_j\to au$ and $T_Ku_j\to T_Ku$ in $L^2(\Omega)$, while the
first part gives convergence of the nonlinear interaction in
$L^{q'}(\Omega)$. The continuous embeddings
$L^2(\Omega)\hookrightarrow X^{-s}(\Omega)$ and
$L^{q'}(\Omega)\hookrightarrow X^{-s}(\Omega)$ complete the proof.
\end{proof}

The local Lipschitz estimate for $\mathcal N_Q$ when $q\geq2$,
including the constant needed for the contraction argument,
will be established in Section~\ref{sec:fixed-point}.

Define the linear form
\begin{equation}\label{eq:linear-bilinear-form}
\mathcal A_s(u,v)
:=
\mathcal E_s(u,v)
-\int_\Omega(T_Ku)v\,dx
-\int_\Omega a(x)uv\,dx,
\qquad u,v\in X_0^s(\Omega).
\end{equation}
It is continuous, with
\begin{equation}\label{eq:linear-form-estimate}
|\mathcal A_s(u,v)|
\leq
\left[
1+C_{S,2}^2
\bigl(|\Omega|\|K\|_\infty+\|a\|_\infty\bigr)
\right]
\|u\|_{X_0^s}\|v\|_{X_0^s}.
\end{equation}
The associated operator
$\mathcal L_s:X_0^s(\Omega)\to X^{-s}(\Omega)$ is defined by
\[
\langle\mathcal L_su,v\rangle:=\mathcal A_s(u,v);
\]
formally, $\mathcal L_su=(-\Delta)^su-T_Ku-au$.

\begin{definition}\label{def:weak-solution}
Let $\lambda\in\mathbb R$ and $f\in X^{-s}(\Omega)$.
A function $u\in X_0^s(\Omega)$ is a weak solution of
problem~\eqref{P} if
\begin{equation}\label{eq:weak-formulation}
\mathcal A_s(u,\varphi)
=
\lambda\int_\Omega\mathcal N_Q(u)\varphi\,dx
+\langle f,\varphi\rangle
\qquad\text{for every }\varphi\in X_0^s(\Omega).
\end{equation}
Equivalently,
$\mathcal L_su=\lambda\mathcal N_Q(u)+f$ in $X^{-s}(\Omega)$.
\end{definition}

For $f\in X^{-s}(\Omega)$, the notation $f\geq0$ means that
$\langle f,\varphi\rangle\geq0$ for every nonnegative
$\varphi\in X_0^s(\Omega)$.
Additional symmetry, sign, coercivity, and smallness assumptions
will be specified in the sections where they are used.
\section{Existence in the competitive symmetric regime}
\label{sec:competitive-regime}

We first consider symmetric kernels and a competitive nonlinear
interaction. Under coercivity of the linear form, direct minimization
yields a weak solution for every forcing term in $X^{-s}(\Omega)$.
Additional sign assumptions ensure that every weak solution is
nonnegative.

\subsection{Variational framework and coercivity}
\label{subsec:variational-framework}

In addition to the standing assumptions of
Section~\ref{sec:preliminaries}, suppose that
\begin{equation}\label{eq:competitive-assumptions}
\begin{gathered}
a\in L^\infty(\Omega),\qquad a\geq0,\qquad a\not\equiv0,\\
K,Q\in L^\infty(\Omega\times\Omega),\qquad
K(x,y)=K(y,x),\qquad Q(x,y)=Q(y,x),\\
Q\geq0\quad\text{a.e. in }\Omega\times\Omega,
\qquad \lambda<0,\qquad f\in X^{-s}(\Omega).
\end{gathered}
\end{equation}
All kernel symmetries are understood almost everywhere.

For $u\in X_0^s(\Omega)$, set
\begin{equation}\label{eq:Q-functional}
\mathcal Q(u)
:=
\iint_{\Omega\times\Omega}
Q(x,y)|u(x)|^q|u(y)|^q\,dx\,dy.
\end{equation}
The kernel bound and the Sobolev embedding give
\begin{equation}\label{eq:Q-functional-bound}
0\leq\mathcal Q(u)
\leq \|Q\|_\infty\|u\|_{L^q(\Omega)}^{2q}
\leq \|Q\|_\infty C_{S,q}^{2q}\|u\|_{X_0^s}^{2q}.
\end{equation}
The associated energy is
\begin{equation}\label{eq:energy-functional}
\mathcal J_\lambda(u)
:=
\frac12\mathcal A_s(u,u)
-\frac{\lambda}{2q}\mathcal Q(u)-\langle f,u\rangle
=
\frac12\mathcal A_s(u,u)
+\frac{|\lambda|}{2q}\mathcal Q(u)-\langle f,u\rangle.
\end{equation}

The symmetry of $K$ makes $\mathcal A_s$ symmetric. Moreover,
the map $u\mapsto |u|^q$ is continuously differentiable from
$L^q(\Omega)$ into $L^1(\Omega)$, and the symmetry of $Q$ gives
\[
\mathcal Q'(u)[v]
=
2q\int_\Omega B_Q(u)|u|^{q-2}uv\,dx.
\]
Consequently, $\mathcal J_\lambda\in C^1(X_0^s(\Omega),\mathbb R)$,
with
\begin{equation}\label{eq:J-derivative}
\langle\mathcal J_\lambda'(u),v\rangle
=
\mathcal A_s(u,v)
-\lambda\int_\Omega\mathcal N_Q(u)v\,dx
-\langle f,v\rangle.
\end{equation}
Thus its critical points are precisely the weak solutions of
problem~\eqref{P}.

The existence argument requires
\begin{equation}\label{eq:direct-coercivity}
\mathcal A_s(u,u)\geq\alpha_0\|u\|_{X_0^s}^{2}
\qquad\text{for every }u\in X_0^s(\Omega),
\end{equation}
for some $\alpha_0>0$. We next give a spectral criterion for
this condition.

Define
\begin{equation}\label{eq:BK-definition}
\mathcal B_{K,s}(u,v)
:=
\mathcal E_s(u,v)-\int_\Omega(T_Ku)v\,dx,
\end{equation}
so that
\begin{equation}\label{eq:A-BK-relation}
\mathcal A_s(u,v)
=
\mathcal B_{K,s}(u,v)-\int_\Omega a(x)uv\,dx.
\end{equation}
A sufficient condition for coercivity of $\mathcal B_{K,s}$ is
\begin{equation}\label{eq:K-smallness}
|\Omega|C_{S,2}^{2}\|K\|_\infty<1.
\end{equation}
Indeed, \eqref{eq:TK-estimate} implies
\begin{equation}\label{eq:BK-coercivity}
\mathcal B_{K,s}(u,u)
\geq
\bigl(1-|\Omega|C_{S,2}^{2}\|K\|_\infty\bigr)
\|u\|_{X_0^s}^{2}.
\end{equation}
Condition~\eqref{eq:K-smallness} is only sufficient; the following
criterion applies whenever $\mathcal B_{K,s}$ is coercive.

\begin{proposition}[A spectral criterion for coercivity]
\label{thm:mu1-implies-coercivity}
Suppose that
$\mathcal B_{K,s}(u,u)\geq\beta_0\|u\|_{X_0^s}^{2}$
for some $\beta_0>0$. Then
\begin{equation}\label{eq:mu1-definition}
\mu_1
:=
\inf_{\substack{u\in X_0^s(\Omega)\\
                 \int_\Omega au^2\,dx>0}}
\frac{\mathcal B_{K,s}(u,u)}
{\displaystyle\int_\Omega a(x)u^2\,dx}
\end{equation}
is finite, positive, and attained by a function $\varphi_1$ satisfying
\[
\int_\Omega a(x)\varphi_1^2\,dx=1.
\]
Moreover,
\begin{equation}\label{eq:weighted-eigenvalue-weak}
\mathcal B_{K,s}(\varphi_1,v)
=
\mu_1\int_\Omega a(x)\varphi_1v\,dx
\qquad\text{for every }v\in X_0^s(\Omega).
\end{equation}
If $\mu_1>1$, then \eqref{eq:direct-coercivity} holds with
\begin{equation}\label{eq:As-coercivity-explicit}
\alpha_0=\beta_0\left(1-\frac1{\mu_1}\right).
\end{equation}
\end{proposition}

\begin{proof}
Since $a\geq0$ and $a\not\equiv0$, there exists
$u\in C_c^\infty(\Omega)$ with $\int_\Omega au^2\,dx>0$.
Thus the admissible set is nonempty and $\mu_1<\infty$.
Coercivity and the $L^2$ embedding give
\[
\mu_1\geq
\frac{\beta_0}{\|a\|_\infty C_{S,2}^{2}}>0.
\]

Choose a minimizing sequence $(u_j)$ normalized by
$\int_\Omega au_j^2\,dx=1$. It is bounded in $X_0^s(\Omega)$,
so, after passing to a subsequence,
\[
u_j\rightharpoonup\varphi_1\quad\text{in }X_0^s(\Omega),
\qquad
u_j\to\varphi_1\quad\text{in }L^2(\Omega).
\]
Consequently, $\int_\Omega a\varphi_1^2\,dx=1$.
The fractional energy is weakly lower semicontinuous, and the
$K$-term is continuous under strong $L^2$ convergence.
Hence $\varphi_1$ attains $\mu_1$.
The Lagrange multiplier rule yields
\eqref{eq:weighted-eigenvalue-weak}, with the multiplier identified
by testing with $\varphi_1$.

Finally, the definition of $\mu_1$ implies
\[
\int_\Omega au^2\,dx
\leq\frac1{\mu_1}\mathcal B_{K,s}(u,u)
\qquad\text{for every }u\in X_0^s(\Omega);
\]
the inequality also holds when its left-hand side is zero.
Therefore, if $\mu_1>1$,
\[
\mathcal A_s(u,u)
\geq
\left(1-\frac1{\mu_1}\right)\mathcal B_{K,s}(u,u)
\geq
\beta_0\left(1-\frac1{\mu_1}\right)\|u\|_{X_0^s}^{2}.
\]
\end{proof}

When $K\equiv0$, the quotient \eqref{eq:mu1-definition} reduces
to the usual weighted first Dirichlet eigenvalue of the restricted
fractional Laplacian. For a general symmetric kernel, no positivity
or simplicity of $\varphi_1$ is needed here.

\subsection{Existence and nonnegativity}
\label{subsec:global-minimizer}

We now assume \eqref{eq:direct-coercivity}, either directly or
through Proposition~\ref{thm:mu1-implies-coercivity}.

\begin{theorem}[Existence of a global minimizer]
\label{thm:competitive-global-minimizer}
Assume \eqref{eq:competitive-assumptions} and
\eqref{eq:direct-coercivity}. Then $\mathcal J_\lambda$ admits
a global minimizer in $X_0^s(\Omega)$, which is a weak solution
of problem~\eqref{P}. If $f\neq0$, every weak solution is nontrivial.
If, in addition, $K\geq0$ almost everywhere and $f\geq0$ in
$X^{-s}(\Omega)$, every weak solution is nonnegative.
\end{theorem}

\begin{proof}
Since $\mathcal Q\geq0$, coercivity gives
\[
\mathcal J_\lambda(u)
\geq
\frac{\alpha_0}{2}\|u\|_{X_0^s}^{2}
-\|f\|_{X^{-s}}\|u\|_{X_0^s}.
\]
Thus $\mathcal J_\lambda$ is bounded below and coercive.

We next verify weak lower semicontinuity.
Let $u_j\rightharpoonup u$ in $X_0^s(\Omega)$.
Since $\mathcal A_s$ is continuous, symmetric, and coercive,
its quadratic form is weakly lower semicontinuous.
The compact embedding into $L^q(\Omega)$ gives
$u_j\to u$ in $L^q(\Omega)$, hence
$|u_j|^q\to|u|^q$ in $L^1(\Omega)$. In particular,
\[
|\mathcal Q(u_j)-\mathcal Q(u)|
\leq
\|Q\|_\infty
\left(\|u_j\|_{L^q}^{q}+\|u\|_{L^q}^{q}\right)
\bigl\||u_j|^q-|u|^q\bigr\|_{L^1}
\longrightarrow0.
\]
Together with $\langle f,u_j\rangle\to\langle f,u\rangle$,
this proves weak lower semicontinuity of $\mathcal J_\lambda$.

The direct method therefore yields a global minimizer
$u_\lambda\in X_0^s(\Omega)$.
By \eqref{eq:J-derivative}, it is a weak solution.
If zero were a weak solution, \eqref{eq:weak-formulation}
would imply $f=0$; thus every weak solution is nontrivial
when $f\neq0$.

Finally, assume $K\geq0$ and $f\geq0$, and let $u$ be any weak
solution. Set $w:=u^-\in X_0^s(\Omega)$ and test with $-w$.
The scalar inequality
\[
(r-t)(r^--t^-)\leq-|r^--t^-|^2
\]
gives $\mathcal E_s(u,-w)\geq\mathcal E_s(w,w)$.
Writing $u=u^+-w$, we also have
\[
\int_\Omega(T_Ku)w\,dx
=
\int_\Omega(T_Ku^+)w\,dx-\int_\Omega(T_Kw)w\,dx
\geq-\int_\Omega(T_Kw)w\,dx.
\]
Since $uw=-w^2$, these inequalities yield
\[
\mathcal A_s(u,-w)\geq\mathcal A_s(w,w).
\]
Moreover, $B_Q(u)\geq0$ and
$\mathcal N_Q(u)(-w)=B_Q(u)w^q$. Hence the weak equation gives
\[
\mathcal A_s(w,w)
+|\lambda|\int_\Omega B_Q(u)w^q\,dx
\leq-\langle f,w\rangle\leq0.
\]
Coercivity implies $w=0$, and therefore $u\geq0$.
\end{proof}

\begin{remark}[The homogeneous problem]
\label{rem:zero-forcing-competitive}
Under the assumptions of
Theorem~\ref{thm:competitive-global-minimizer}, if $f=0$, then
zero is the unique weak solution. Indeed, testing the equation
with $u$ gives
\[
\mathcal A_s(u,u)+|\lambda|\mathcal Q(u)=0,
\]
and coercivity forces $u=0$.
Thus a nonzero forcing term is necessary for a nontrivial
solution in this regime.
\end{remark}
\section{A fixed-point approach in the general regime}
\label{sec:fixed-point}

We now allow arbitrary $\lambda\in\mathbb R$ and nonsymmetric,
sign-changing kernels. When $Q\geq0$, this includes both competitive
and cooperative interactions; more generally, the interaction may
have no fixed sign, and no variational structure is required.
Under suitable smallness conditions, we establish existence by
Schauder's theorem and, for $q\geq2$, local uniqueness by the
contraction principle.

Throughout this section, we assume
\begin{equation}\label{eq:general-assumptions}
0<s<1,\qquad 1<q<2_s^*,\qquad \lambda\in\mathbb R,
\qquad a\in L^\infty(\Omega),\qquad
K,Q\in L^\infty(\Omega\times\Omega),\qquad
f\in X^{-s}(\Omega).
\end{equation}
No symmetry or sign assumptions are imposed unless explicitly stated.

\subsection{Schauder existence and nonnegativity}
\label{subsec:existence-schauder}

For $u\in X_0^s(\Omega)$, define $G(u)\in X^{-s}(\Omega)$ by
\begin{equation}\label{eq:G-definition}
\langle G(u),\varphi\rangle
:=
\int_\Omega\bigl(au+T_Ku+\lambda\mathcal N_Q(u)\bigr)\varphi\,dx
+\langle f,\varphi\rangle,
\qquad \varphi\in X_0^s(\Omega).
\end{equation}
Set
\begin{equation}\label{eq:eta-gamma-definition}
\eta
:=
C_{S,2}^{2}\bigl(\|a\|_\infty+|\Omega|\|K\|_\infty\bigr),
\qquad
\gamma
:=
|\lambda|\|Q\|_\infty C_{S,q}^{2q}.
\end{equation}
The estimates of Section~\ref{sec:preliminaries} give
\begin{equation}\label{eq:G-estimate}
\|G(u)\|_{X^{-s}}
\leq
\eta\|u\|_{X_0^s}
+\gamma\|u\|_{X_0^s}^{2q-1}
+\|f\|_{X^{-s}}.
\end{equation}

For each $u\in X_0^s(\Omega)$, the Lax--Milgram theorem provides
a unique $S(u)\in X_0^s(\Omega)$ satisfying
\begin{equation}\label{eq:def-S-schauder}
\mathcal E_s(S(u),\varphi)
=
\langle G(u),\varphi\rangle
\qquad\text{for every }\varphi\in X_0^s(\Omega).
\end{equation}
In particular,
\begin{equation}\label{eq:S-basic-estimate-schauder}
\|S(u)\|_{X_0^s}\leq\|G(u)\|_{X^{-s}},
\qquad
\|S(u)-S(v)\|_{X_0^s}
\leq\|G(u)-G(v)\|_{X^{-s}}.
\end{equation}
By construction, the fixed points of $S$ are precisely the weak
solutions of problem~\eqref{P}.

\begin{lemma}[Compactness of the solution operator]
\label{prop:S-schauder-properties}
The map $S:X_0^s(\Omega)\to X_0^s(\Omega)$ is continuous and
maps bounded sets into relatively compact sets.
\end{lemma}

\begin{proof}
Lemma~\ref{lem:interaction-compactness} implies that $G$ is
continuous and that
\[
u_j\rightharpoonup u\quad\text{in }X_0^s(\Omega)
\quad\Longrightarrow\quad
G(u_j)\to G(u)\quad\text{in }X^{-s}(\Omega).
\]
The continuity of $S$ follows from
\eqref{eq:S-basic-estimate-schauder}.
Every bounded sequence in $X_0^s(\Omega)$ has a weakly convergent
subsequence, and the same estimate shows that its image under $S$
converges strongly in $X_0^s(\Omega).$
\end{proof}

For $r>0$, let
\begin{equation}\label{eq:Er-schauder}
E_r
:=
\left\{
u\in X_0^s(\Omega):\|u\|_{X_0^s}\leq r
\right\}.
\end{equation}

\begin{theorem}[Existence in an invariant ball]
\label{thm:existence-general-schauder}
Assume \eqref{eq:general-assumptions}, let $\eta<1$, and suppose
that there exists $r>0$ such that
\begin{equation}\label{eq:invariant-ball-condition}
\gamma r^{2q-2}+\frac{\|f\|_{X^{-s}}}{r}
\leq1-\eta.
\end{equation}
Then problem~\eqref{P} admits a weak solution $u^*\in E_r$.
If $f\neq0$, this solution is nontrivial.
\end{theorem}

\begin{proof}
For $u\in E_r$, estimates \eqref{eq:G-estimate} and
\eqref{eq:S-basic-estimate-schauder} yield
\[
\|S(u)\|_{X_0^s}
\leq\eta r+\gamma r^{2q-1}+\|f\|_{X^{-s}}
\leq r.
\]
Thus $S(E_r)\subseteq E_r$.
Since $E_r$ is nonempty, closed, convex, and bounded, the
compactness and continuity of $S$ allow us to apply Schauder's
fixed-point theorem. The resulting fixed point is a weak solution
by \eqref{eq:def-S-schauder}. Finally, zero can satisfy the weak
equation only when $f=0$.
\end{proof}

An explicit sufficient condition is obtained as follows.
If $\gamma>0$, choose
\begin{equation}\label{eq:explicit-invariant-radius}
r_0
:=
\left(\frac{1-\eta}{2\gamma}\right)^{1/(2q-2)}.
\end{equation}
Then \eqref{eq:invariant-ball-condition} holds whenever
\begin{equation}\label{eq:explicit-forcing-smallness}
\|f\|_{X^{-s}}
\leq\frac{1-\eta}{2}\,r_0.
\end{equation}
Thus, for every fixed $\lambda$, the theorem applies to sufficiently
small forcing terms, provided $\eta<1$. The radius $r_0$ is a
convenient choice rather than an optimization of the sufficient bound.

If $\gamma=0$, the nonlinear term vanishes, and any $r_0>0$ satisfying
\begin{equation}\label{eq:linear-invariant-radius}
\|f\|_{X^{-s}}\leq(1-\eta)r_0
\end{equation}
is admissible. In this case, $S$ is a contraction on the whole space
with constant at most $\eta$, so the weak solution is globally unique
for every $f\in X^{-s}(\Omega)$.

\begin{remark}\label{rem:fixed-point-coercivity}
Although symmetry is not required, the condition $\eta<1$ implies
\begin{equation}\label{eq:general-coercivity}
\mathcal A_s(u,u)
\geq(1-\eta)\|u\|_{X_0^s}^{2}
\qquad\text{for every }u\in X_0^s(\Omega).
\end{equation}
The fixed-point approach therefore accommodates nonsymmetric kernels
and arbitrary signs of $\lambda$, but the present sufficient condition
still ensures coercivity of the linear form.
\end{remark}

\begin{proposition}[Nonnegative solutions]
\label{prop:nonnegative-schauder-solution}
Assume the hypotheses of
Theorem~\ref{thm:existence-general-schauder}, together with
\[
K,Q\geq0\quad\text{a.e. in }\Omega\times\Omega,
\qquad f\geq0\quad\text{in }X^{-s}(\Omega).
\]
If $\lambda\geq0$ and $a\geq0$, there exists a nonnegative weak
solution in $E_r$. If $\lambda<0$, every weak solution is
nonnegative, without any sign restriction on $a$.
In either case, the solution is nontrivial whenever $f\neq0$.
\end{proposition}

\begin{proof}
Suppose first that $\lambda\geq0$ and $a\geq0$. Consider the
nonempty closed convex set
\[
E_r^+
:=
\{u\in E_r:u\geq0\text{ a.e. in }\Omega\}.
\]
For $u\in E_r^+$, all terms defining $G(u)$ are nonnegative
when paired with nonnegative test functions.
Writing $v=S(u)$ and testing \eqref{eq:def-S-schauder} with $v^-$,
we obtain
\[
0\leq\langle G(u),v^-\rangle
=\mathcal E_s(v,v^-)
\leq-\|v^-\|_{X_0^s}^{2}.
\]
Hence $v\geq0$. The invariant-ball estimate also gives $v\in E_r$,
so $S(E_r^+)\subseteq E_r^+$.
Schauder's theorem on $E_r^+$ yields a nonnegative weak solution.

Now suppose that $\lambda<0$, and let $u$ be any weak solution.
Set $w=u^-$ and test the equation with $-w$.
The fractional truncation inequality gives
$\mathcal E_s(u,-w)\geq\mathcal E_s(w,w)$.
Moreover, since $K\geq0$,
\[
\int_\Omega(T_Ku)w\,dx
=
\int_\Omega(T_Ku^+)w\,dx-\int_\Omega(T_Kw)w\,dx
\geq-\int_\Omega(T_Kw)w\,dx.
\]
Using $uw=-w^2$, we obtain
$\mathcal A_s(u,-w)\geq\mathcal A_s(w,w)$.
Consequently, \eqref{eq:general-coercivity} and the weak equation imply
\[
(1-\eta)\|w\|_{X_0^s}^{2}
+|\lambda|\int_\Omega B_Q(u)w^q\,dx
\leq-\langle f,w\rangle\leq0.
\]
Since $B_Q(u)\geq0$, this forces $w=0$.
The nontriviality assertion follows from
Theorem~\ref{thm:existence-general-schauder}.
\end{proof}

\subsection{Local uniqueness and Picard iteration}
\label{subsec:uniqueness-banach}

For the contraction argument, assume $2\leq q<2_s^*$ and define
\begin{equation}\label{eq:CqQS}
C_{q,Q,S}
:=
\|Q\|_\infty
\left[(q-1)2^{q-2}+q2^{q-1}\right]C_{S,q}^{2q}.
\end{equation}

\begin{lemma}[Local Lipschitz estimate]
\label{lem:NQ-local-Lipschitz}
For every $r>0$ and every $u,v\in E_r$,
\begin{equation}\label{eq:NQ-local-Lipschitz}
\|\mathcal N_Q(u)-\mathcal N_Q(v)\|_{X^{-s}}
\leq C_{q,Q,S}r^{2q-2}\|u-v\|_{X_0^s}.
\end{equation}
\end{lemma}

\begin{proof}
Decompose the difference as
\[
\mathcal N_Q(u)-\mathcal N_Q(v)
=
\bigl(|u|^{q-2}u-|v|^{q-2}v\bigr)B_Q(u)
+|v|^{q-2}v\bigl(B_Q(u)-B_Q(v)\bigr).
\]
For $q\geq2$, the inequalities
\[
\bigl||\xi|^{q-2}\xi-|\zeta|^{q-2}\zeta\bigr|
\leq(q-1)(|\xi|+|\zeta|)^{q-2}|\xi-\zeta|
\]
and
\[
\bigl||\xi|^q-|\zeta|^q\bigr|
\leq q(|\xi|+|\zeta|)^{q-1}|\xi-\zeta|
\]
hold for real $\xi,\zeta$, with the first inequality interpreted
directly as $|\xi-\zeta|\leq|\xi-\zeta|$ when $q=2$.
H\"older's inequality and
$\|u\|_{L^q},\|v\|_{L^q}\leq C_{S,q}r$ give
\begin{align*}
\|\mathcal N_Q(u)-\mathcal N_Q(v)\|_{L^{q'}}
&\leq
\|Q\|_\infty
\left[(q-1)2^{q-2}+q2^{q-1}\right]\\
&\qquad{}\times
(C_{S,q}r)^{2q-2}\|u-v\|_{L^q},
\end{align*}
where $q'=q/(q-1)$. Pairing with a test function in $L^q(\Omega)$
and applying \eqref{eq:fractional-embedding} twice proves
\eqref{eq:NQ-local-Lipschitz}.
\end{proof}

\begin{theorem}[Uniqueness in an invariant ball]
\label{thm:uniqueness-banach}
Assume \eqref{eq:general-assumptions}, let $2\leq q<2_s^*$,
and suppose that, for some $r>0$,
\begin{equation}\label{eq:invariance-and-contraction}
\gamma r^{2q-2}+\frac{\|f\|_{X^{-s}}}{r}\leq1-\eta,
\end{equation}
and
\begin{equation}\label{eq:contraction-condition}
\kappa_r
:=
\eta+|\lambda|C_{q,Q,S}r^{2q-2}<1.
\end{equation}
Then problem~\eqref{P} has a unique weak solution $u^*\in E_r$.
For every $u_0\in E_r$, the Picard iterates
\[
u_{n+1}:=S(u_n),\qquad n\geq0,
\]
remain in $E_r$ and converge strongly to $u^*$ in $X_0^s(\Omega)$.
If $0<\kappa_r<1$, they satisfy
\begin{equation}\label{eq:Picard-error-estimate}
\|u_n-u^*\|_{X_0^s}
\leq
\frac{\kappa_r^n}{1-\kappa_r}
\|u_1-u_0\|_{X_0^s},
\qquad n\geq0.
\end{equation}
If $\kappa_r=0$, the iteration reaches $u^*$ after one step.
\end{theorem}

\begin{proof}
Condition~\eqref{eq:invariance-and-contraction} ensures
$S(E_r)\subseteq E_r$. For $u,v\in E_r$,
\eqref{eq:S-basic-estimate-schauder} and
Lemma~\ref{lem:NQ-local-Lipschitz} give
\begin{align*}
\|S(u)-S(v)\|_{X_0^s}
&\leq
\eta\|u-v\|_{X_0^s}
+|\lambda|
\|\mathcal N_Q(u)-\mathcal N_Q(v)\|_{X^{-s}}\\
&\leq\kappa_r\|u-v\|_{X_0^s}.
\end{align*}
Thus $S$ is a contraction on the complete metric space $E_r$.
The Banach fixed-point theorem yields the unique fixed point and
convergence of the iterates.
For $0<\kappa_r<1$, summing
\[
\|u_{n+1}-u_n\|_{X_0^s}
\leq\kappa_r^n\|u_1-u_0\|_{X_0^s}
\]
gives \eqref{eq:Picard-error-estimate}.
When $\kappa_r=0$, the map $S$ is constant on $E_r$, and its
value is the fixed point.
\end{proof}

\begin{remark}\label{rem:uniqueness-in-ball}
The uniqueness assertion concerns solutions in $E_r$.
It becomes global only if an independent a priori estimate places
every weak solution in this ball.
The contraction argument is stated for $q\geq2$; the Schauder
existence result remains valid throughout $1<q<2_s^*$.
\end{remark}
\section{The local limit as \texorpdfstring{$s\uparrow1$}{s approaches 1}}
\label{sec:local-limit}

We study the solutions constructed in Section~\ref{sec:fixed-point}
as the fractional order approaches one. Uniform estimates yield
subsequential convergence to a solution of the corresponding local
Dirichlet problem. Testing the equations with the solutions also
gives convergence of the fractional energies to the Dirichlet energy.

Throughout this section, $\Omega\subset\mathbb R^N$, $N\geq2$, is a
bounded Lipschitz domain. The data
\begin{equation}\label{eq:data-limit-section}
\lambda\in\mathbb R,\qquad
a\in L^\infty(\Omega),\qquad
K,Q\in L^\infty(\Omega\times\Omega),\qquad
f\in L^2(\Omega)
\end{equation}
are fixed independently of $s$. We assume $1<q<2^*$, where
\[
2^*
:=
\begin{cases}
\dfrac{2N}{N-2},&N>2,\\[2mm]
+\infty,&N=2.
\end{cases}
\]
Choose $s_0\in(0,1)$ sufficiently close to one that
\begin{equation}\label{eq:q-below-fractional-critical}
q<2_{s_0}^*.
\end{equation}
For $s\in[s_0,1)$, consider
\begin{equation}\label{eq:s-dependent-problem}
\begin{cases}
(-\Delta)^s u_s
=au_s+T_Ku_s+\lambda\mathcal N_Q(u_s)+f
&\text{in }\Omega,\\
u_s=0&\text{in }\mathbb R^N\setminus\Omega.
\end{cases}
\end{equation}
We use the normalized form $\mathcal E_s$ from
\eqref{eq:fractional-bilinear-form} and write
$\|w\|_s:=\mathcal E_s(w,w)^{1/2}$.

\subsection{Uniform estimates and compactness}
\label{subsec:uniform-invariant-balls}

For every fixed $\sigma\in(0,1)$ and $1\leq r<2_\sigma^*$,
there exists $C_{\sigma,r}>0$, independent of
$s\in[\sigma,1)$, such that
\[
\|w\|_{L^r(\Omega)}
\leq C_{\sigma,r}\|w\|_s,
\qquad w\in X_0^s(\Omega).
\]
These uniform estimates follow from the normalized fractional
Poincar\'e inequality and the Sobolev embedding at a fixed order.
Indeed, using the unitary Fourier transform of the zero extension,
\[
\mathcal E_s(w,w)
=
\int_{\mathbb R^N}|\xi|^{2s}|\widehat w(\xi)|^2\,d\xi,
\]
and $|\xi|^{2\sigma}\leq1+|\xi|^{2s}$ for $s\geq\sigma$.
Thus the uniform $L^2$ estimate controls the $H^\sigma(\mathbb R^N)$
norm by $\|w\|_s$, and the fixed-order Sobolev embedding applies.
We refer to
\cite{DiNezzaPalatucciValdinoci2012,BrascoPariniSquassina2016}
for the underlying fractional Sobolev and local-limit results.

In particular, there are constants $C_2,C_q>0$, independent of
$s\in[s_0,1)$, such that
\begin{equation}\label{eq:uniform-embeddings}
\|w\|_{L^2(\Omega)}\leq C_2\|w\|_s,
\qquad
\|w\|_{L^q(\Omega)}\leq C_q\|w\|_s.
\end{equation}
The same constants can be used in the local inequalities
\begin{equation}\label{eq:local-embedding-constants}
\|w\|_{L^2(\Omega)}\leq C_2\|\nabla w\|_{L^2(\Omega)},
\qquad
\|w\|_{L^q(\Omega)}\leq C_q\|\nabla w\|_{L^2(\Omega)},
\quad w\in H_0^1(\Omega).
\end{equation}
To see this, first take $w\in C_c^\infty(\Omega)$, let $s\uparrow1$
in \eqref{eq:uniform-embeddings}, and then use density in
$H_0^1(\Omega)$.

Set
\begin{equation}\label{eq:uniform-constants}
A:=C_2^2\bigl(\|a\|_\infty+|\Omega|\|K\|_\infty\bigr),
\qquad
B:=|\lambda|\|Q\|_\infty C_q^{2q},
\qquad
D:=C_2\|f\|_{L^2(\Omega)}.
\end{equation}

\begin{proposition}[Uniform existence bound]
\label{prop:uniform-invariant-ball}
Assume that $A<1$ and that, for some $r_0>0$,
\begin{equation}\label{eq:uniform-invariant-condition}
Br_0^{2q-2}+\frac{D}{r_0}\leq1-A.
\end{equation}
Then, for every $s\in[s_0,1)$,
problem~\eqref{eq:s-dependent-problem} has a weak solution
$u_s\in X_0^s(\Omega)$ satisfying
\begin{equation}\label{eq:uniform-energy-bound}
\mathcal E_s(u_s,u_s)\leq r_0^2.
\end{equation}
\end{proposition}

\begin{proof}
By \eqref{eq:uniform-embeddings},
$\|f\|_{X^{-s}}\leq D$, while the constants $\eta$ and $\gamma$
in Section~\ref{sec:fixed-point} can be bounded by $A$ and $B$,
respectively. Thus \eqref{eq:uniform-invariant-condition} verifies
the invariant-ball condition of
Theorem~\ref{thm:existence-general-schauder} for every
$s\in[s_0,1)$, with the same radius $r_0$.
\end{proof}

For $B>0$, an explicit sufficient choice is
\begin{equation}\label{eq:explicit-uniform-radius}
r_0
=
\left(\frac{1-A}{2B}\right)^{1/(2q-2)},
\qquad
D\leq\frac{1-A}{2}\,r_0.
\end{equation}
If $B=0$, any $r_0>0$ satisfying $D\leq(1-A)r_0$ is admissible.
These conditions are uniform in $s$, although the explicit radius
is not intended to optimize the sufficient forcing bound.

\begin{lemma}[Compactness as the order approaches one]
\label{prop:compactness-us}
Let $s_n\uparrow1$ and let $w_n\in X_0^{s_n}(\Omega)$ satisfy
$\sup_n\|w_n\|_{s_n}<\infty$.
Then, after passing to a subsequence, there exists
$w\in H_0^1(\Omega)$ such that
\begin{equation}\label{eq:strong-Lp-all}
w_n\to w
\quad\text{strongly in }L^p(\Omega)
\quad\text{for every }1\leq p<2^*,
\end{equation}
and
\begin{equation}\label{eq:BBM-liminf}
\int_\Omega|\nabla w|^2\,dx
\leq
\liminf_{n\to\infty}\mathcal E_{s_n}(w_n,w_n).
\end{equation}
The same conclusion holds for any sequence $s_n\to1^-$,
without a monotonicity assumption.
\end{lemma}

\begin{proof}
The uniform $L^2$ estimate and the normalized
Bourgain--Brezis--Mironescu compactness theorem for zero extensions
give, up to a subsequence,
\[
w_n\to w\quad\text{in }L^2(\Omega),
\qquad w\in H_0^1(\Omega),
\]
together with \eqref{eq:BBM-liminf}; see
\cite{BourgainBrezisMironescu2001,BrascoPariniSquassina2016}.
The zero exterior condition passes to the limit in $L^2(\mathbb R^N)$;
the Lipschitz regularity of $\Omega$ identifies the restriction of
this $H^1(\mathbb R^N)$ limit with an element of $H_0^1(\Omega)$.

For $2<p<2^*$, choose $p<\rho<2^*$ and then
$\sigma<1$ sufficiently close to one that $\rho<2_\sigma^*$.
For all sufficiently large $n$, $s_n\geq\sigma$, and the uniform
embedding gives a bounded sequence in $L^\rho(\Omega)$.
Interpolation with the strong $L^2$ convergence yields strong
$L^p$ convergence. For $1\leq p<2$, the conclusion follows from
the boundedness of $\Omega$. This argument uses the same
$L^2$-convergent subsequence for every $p<2^*$.
\end{proof}

\subsection{Convergence to the local problem}
\label{subsec:passage-to-limit}

The limiting equation is
\begin{equation}\label{eq:limiting-local-problem}
\begin{cases}
-\Delta u=au+T_Ku+\lambda\mathcal N_Q(u)+f
&\text{in }\Omega,\\
u=0&\text{on }\partial\Omega,
\end{cases}
\end{equation}
understood weakly in $H_0^1(\Omega)$.

\begin{theorem}[Convergence of solutions and energies]
\label{thm:main-local-limit}
Assume \eqref{eq:data-limit-section},
\eqref{eq:q-below-fractional-critical}, $A<1$, and
\eqref{eq:uniform-invariant-condition}.
For each $s\in[s_0,1)$, choose a weak solution $u_s$ satisfying
\eqref{eq:uniform-energy-bound}.
Then, for every sequence $s_n\to1^-$, there exist a subsequence
and a weak solution $u\in H_0^1(\Omega)$ of
\eqref{eq:limiting-local-problem} such that
\[
u_{s_n}\to u
\quad\text{strongly in }L^p(\Omega)
\quad\text{for every }1\leq p<2^*.
\]
Moreover,
\begin{equation}\label{eq:local-limit-energy-convergence}
\mathcal E_{s_n}(u_{s_n},u_{s_n})
\longrightarrow
\int_\Omega|\nabla u|^2\,dx
\leq r_0^2.
\end{equation}
If $f\neq0$, the limit $u$ is nontrivial.
If the selected solutions $u_s$ are nonnegative, then $u\geq0$.
\end{theorem}

\begin{proof}
Write $u_n:=u_{s_n}$.
Lemma~\ref{prop:compactness-us} provides a subsequence and
$u\in H_0^1(\Omega)$ with the asserted strong convergence and
$\|\nabla u\|_{L^2}\leq r_0$.

We first pass to the limit in the weak equation.
For $\varphi\in C_c^\infty(\Omega)$, its zero extension is smooth
and compactly supported in $\mathbb R^N$. The Fourier representation
of the normalized fractional Laplacian gives
\[
(-\Delta)^{s_n}\varphi\to-\Delta\varphi
\quad\text{in }L^2(\mathbb R^N).
\]
Using strong $L^2$ convergence of the zero extensions, we obtain
\[
\mathcal E_{s_n}(u_n,\varphi)
=
\int_{\mathbb R^N}u_n(-\Delta)^{s_n}\varphi\,dx
\longrightarrow
\int_\Omega\nabla u\cdot\nabla\varphi\,dx.
\]
The linear terms converge by strong $L^2$ convergence and
\eqref{eq:TK-estimate}. Since $q<2^*$, we also have
$u_n\to u$ in $L^q(\Omega)$. The $L^q$ continuity established in
Lemma~\ref{lem:interaction-compactness} therefore gives
\[
B_Q(u_n)\to B_Q(u)\quad\text{in }L^\infty(\Omega),
\qquad
\mathcal N_Q(u_n)\to\mathcal N_Q(u)
\quad\text{in }L^{q'}(\Omega),
\]
where $q'=q/(q-1)$.
Passing to the limit yields
\begin{equation}\label{eq:limiting-weak-equation}
\int_\Omega\nabla u\cdot\nabla\varphi\,dx
=
\int_\Omega\bigl(au+T_Ku+\lambda\mathcal N_Q(u)+f\bigr)
\varphi\,dx.
\end{equation}
Both sides are continuous on $H_0^1(\Omega)$, by the Sobolev
embedding and $q<2^*$. Density extends this identity to every
$\varphi\in H_0^1(\Omega)$.

To prove energy convergence, retain the notation
\[
\mathcal Q(v)
=
\iint_{\Omega\times\Omega}
Q(x,y)|v(x)|^q|v(y)|^q\,dx\,dy.
\]
This expression is well defined for the present kernels,
without symmetry or sign assumptions.
Testing the fractional equation with $u_n$ gives
\begin{equation}\label{eq:fractional-solution-energy-identity}
\mathcal E_{s_n}(u_n,u_n)
=
\int_\Omega au_n^2\,dx
+\int_\Omega(T_Ku_n)u_n\,dx
+\lambda\mathcal Q(u_n)
+\int_\Omega fu_n\,dx.
\end{equation}
The linear terms converge by strong $L^2$ convergence.
Furthermore, strong $L^q$ convergence gives
$|u_n|^q\to|u|^q$ in $L^1(\Omega)$, and
\[
|\mathcal Q(u_n)-\mathcal Q(u)|
\leq
\|Q\|_\infty
\left(\|u_n\|_{L^q}^{q}+\|u\|_{L^q}^{q}\right)
\bigl\||u_n|^q-|u|^q\bigr\|_{L^1}
\longrightarrow0.
\]
Thus the right-hand side of
\eqref{eq:fractional-solution-energy-identity} converges to
\[
\int_\Omega au^2\,dx
+\int_\Omega(T_Ku)u\,dx
+\lambda\mathcal Q(u)+\int_\Omega fu\,dx.
\]
Testing \eqref{eq:limiting-weak-equation} with $u$ identifies this
quantity with $\int_\Omega|\nabla u|^2\,dx$, proving
\eqref{eq:local-limit-energy-convergence}.

If $u=0$, the limiting weak equation implies $f=0$, so the
limit is nontrivial when $f\neq0$.
Finally, nonnegativity is preserved under strong $L^2$ convergence.
\end{proof}

\begin{corollary}[Convergence of the entire family]
\label{cor:whole-family-convergence}
Under the assumptions of Theorem~\ref{thm:main-local-limit},
suppose that the local problem has a unique weak solution in
\[
E_1(r_0)
:=
\left\{
w\in H_0^1(\Omega):
\|\nabla w\|_{L^2(\Omega)}\leq r_0
\right\}.
\]
Then the entire selected family satisfies
\[
u_s\to u
\quad\text{strongly in }L^p(\Omega)
\quad\text{for every }1\leq p<2^*,
\]
and
\[
\mathcal E_s(u_s,u_s)
\longrightarrow\int_\Omega|\nabla u|^2\,dx
\qquad\text{as }s\uparrow1.
\]

In particular, this conclusion holds if $q\geq2$ and
\begin{equation}\label{eq:uniform-contraction-local-limit}
A+|\lambda|\widehat C_{q,Q}r_0^{2q-2}<1,
\qquad
\widehat C_{q,Q}
:=
\|Q\|_\infty
\left[(q-1)2^{q-2}+q2^{q-1}\right]C_q^{2q}.
\end{equation}
\end{corollary}

\begin{proof}
Every sequence $s_n\to1^-$ has a subsequence converging to a
local weak solution in $E_1(r_0)$.
Uniqueness in this ball identifies every such limit with $u$.
The usual subsequence contradiction argument then proves
convergence of the entire family in each $L^p(\Omega)$.
The energy identity
\eqref{eq:fractional-solution-energy-identity} gives convergence
of the entire family of energies as well.

For the last assertion, define the local solution operator by
solving
\[
-\Delta v=aw+T_Kw+\lambda\mathcal N_Q(w)+f,
\qquad v\in H_0^1(\Omega).
\]
The local estimates \eqref{eq:local-embedding-constants} and
\eqref{eq:uniform-invariant-condition} show that this operator
maps $E_1(r_0)$ into itself.
The proof of Lemma~\ref{lem:NQ-local-Lipschitz}, using $C_q$
in place of $C_{S,q}$, bounds its Lipschitz constant by
\[
A+|\lambda|\widehat C_{q,Q}r_0^{2q-2}.
\]
Condition~\eqref{eq:uniform-contraction-local-limit} therefore
gives uniqueness in $E_1(r_0)$ by the Banach fixed-point theorem.
\end{proof}

For nonnegative symmetric kernels, the linear interaction can also
be interpreted through additional interior jumps of the killed
stable process. We leave the development of this probabilistic
interpretation and its fractional-to-local limit to future work.

\section*{Funding}

The authors declare that no funds, grants, or other support were received

during the preparation of this manuscript.

  \bibliographystyle{plain}
\bibliography{bibliography.bib}
\end{document}